\documentclass[11pt]{amsart}
\usepackage{graphicx,amssymb,amsmath,amsthm}
\usepackage{enumerate}
\usepackage{dsfont}
\usepackage{mathrsfs}
\usepackage{verbatim}
\usepackage{epsfig}
\usepackage{lscape}
\usepackage{subfigure}
\usepackage[colorlinks, linkcolor=red, anchorcolor=blue, citecolor=green]{hyperref}

\usepackage{epstopdf}
\usepackage{color}
\usepackage{caption}
\usepackage{algorithm}
\usepackage{algpseudocode}

\newcommand{\abs}[1]{\lvert#1\rvert}
\newcommand{\absinn}[1]{\vert\langle {#1} \rangle\rvert}
\newcommand{\argmin}[1]{\mathop{\rm argmin}\limits_{#1}}
\newcommand{\argmax}[1]{\mathop{\rm argmax}\limits_{#1}}

\newcommand{\R}{{\mathbb R}}

\newcommand{\C}{{\mathbb C}}

\newcommand{\vx}{{\mathbf x}}
\newcommand{\vy}{{\mathbf y}}

\newcommand{\be}{{\mathbf e}}
\newcommand{\bl}{{\mathbf 0}}

\newcommand{\FP}{{\rm FP}}

\newcommand{\Z}{{\mathbb Z}}

\newcommand{\rank}{{\rm rank}}

\renewcommand{\omega}{\eta}

\newcommand{\RNum}[1]{\uppercase\expandafter{\romannumeral #1\relax}}

\newtheorem{definition}{Definition}[section]

\newtheorem{theorem}[definition]{Theorem}
\newtheorem{lemma}[definition]{Lemma}
\newtheorem{conjecture}[definition]{Conjecture}

\newtheorem{remark}[definition]{Remark}

\date{}

\begin{document}
\baselineskip 18pt
\bibliographystyle{plain}
\title{The Minimizers of the $p$-frame potential}

\author{Zhiqiang Xu}
\thanks{The project  was supported  by NSFC grant (91630203, 11688101),
Beijing Natural Science Foundation (Z180002).}

\address{LSEC, Inst.~Comp.~Math., Academy of
Mathematics and System Science,  Chinese Academy of Sciences, Beijing, 100091, China
\newline
School of Mathematical Sciences, University of Chinese Academy of Sciences, Beijing 100049, China}
\email{xuzq@lsec.cc.ac.cn}

\author{Zili Xu}
\address{ LSEC, Inst.~Comp.~Math., Academy of
Mathematics and System Science,  Chinese Academy of Sciences, Beijing, 100091, China}
\email{ xuzili@lsec.cc.ac.cn}
\keywords{Frame potential, Tight frames, Spherical designs }
\begin{abstract}
For any positive real number $p$, the $p$-frame potential of $N$
unit vectors $X:=\{\vx_1,\ldots,\vx_N\}\subset \R^d$ is defined as
$\FP_{p,N,d}(X)=\sum_{i\neq j}\absinn{\vx_i,\vx_j}^p$.  In this paper,
we focus on this quantity for $N=d+1$ points and establish uniqueness of minimizers of $\FP_{p,d+1,d}$ for all   $p\in (0,2)$. Our results completely solve the minimization problem
 of the $p$-frame potential when $N=d+1$, confirming a conjecture posed
 by Chen, Gonzales, Goodman, Kang and Okoudjou \cite{Chen}.

\end{abstract}

\maketitle

\section{Introduction}
\subsection{The $p$-frame potential}

Minimal potential energy problems have been actively discussed over the last decades in connection with  applications in physics, signal analysis and numerical integration. Generally, one aims to find distributions of $N$ points on the unit sphere which minimize the potential energy over all sized $N$ configurations \cite{y, Cohn, Saff}.

One of the most interesting potential energies is the $p$-frame potential. Assume that $X:=\{\vx_i\}_{i=1}^{N}$  where  $\vx_i\in \R^d$ with $\|\vx_i\|_2=1, i=1,\ldots,N$.
For $p>0$, the value
\begin{equation}\label{e2}
\FP_{p,N,d}(X):=\sum\limits_{i=1}^{N}\sum\limits_{j\neq i}\absinn{\vx_i,\vx_j}^p,
\end{equation}
 is called the {\em $p$-frame potential } of $X$ (see  \cite{Ehler, Chen}). The minimization problem  of the $p$-frame potential is to solve
\begin{equation}\label{eq:e2}
\argmin{X\in S(N,d)} \FP_{p,N,d}(X),
\end{equation}
where
$
S(N,d)
$
denotes all sets of $N$ unit-norm vectors in $\R^d$. This problem actually has a long history and has attracted much attention.
For $N\leq d$, the set of $N$ orthogonal vectors in $\R^d$ is always the minimizer of \eqref{eq:e2} for any positive $p$ and hence we only consider the case where  $N\geq d+1$. We also note that the value of $\FP_{p,N,d}(X)$ does  not change if we replace $\vx_i$ by $c_iU\vx_i$ for each $i\in\{1,2,\cdots,N\}$, where $U$ is an orthogonal matrix and $c_i\in\{1,-1\}$. Thus, for convenience  we say the minimizer of (\ref{eq:e2}) is unique if the solution to
 (\ref{eq:e2}) is unique up to a common orthogonal transformation and a real unimodular constant for each vector.

\subsection{Related work}

There are many results which presented a lower bound of $\FP_{p,N,d}(X)$.
  When $p=2k\in {2\mathbb Z}_+$, the following bound
\begin{equation}\label{e4}
\FP_{2k,N,d}(X)\geq N^2\frac{1\cdot 3\cdot 5\dots (2k-1)}{d(d+2)\dots (d+2k-2)}-N
\end{equation}
was  presented by Sidelnikov in \cite{Sidelnikov}. The equality in (\ref{e4}) holds precisely
when $X\cup-X$ is a  spherical $2k$-design, see \cite{Goethals,Sidelnikov}. The set $X=\{\vx_i\}_{i=1}^{N}$ attaining the bound in (\ref{e4}) can also be identified as a projective $k$-design since we consider $\vx_i$ and $-\vx_i$ as the same point (see \cite{Hoggar} for the definition of projective designs).
  For the special case $k=1$, i.e. $p=2$, the minimizers of the $2$-frame
potential are projective $1$-designs \cite{Goethals,Sidelnikov}. Noting the fact that
finite unit-norm tight frames (FUNTFs) are in one-to-one correspondence with
projective 1-designs \cite{Delsarte}, we know that the equality in (\ref{e4}) holds
for $k=1$ and all $N\geq d$ as long as $X$ is a FUNTF (see also \cite{Fickus}).
However, when $k>1$, the bound in (\ref{e4}) is not tight for small $N$ since the
existence of spherical designs requires $N$ to be large enough \cite{Delsarte}.

For any $p>2$,  Ehler and Okoudjou provided another bound in \cite{Ehler}:
\begin{equation}\label{e5}
\FP_{p,N,d}(X) \geq N(N-1)\left(\frac{N-d}{d(N-1)}\right)^{\frac{p}{2}},
\end{equation}
where equality holds if and only if $X$ is an equiangular tight frame (ETF) in $\R^d$
\cite{ETF, Holmes}. We take $N=d+1$ as an example. Since there always exist $d+1$
unit vectors in $\R^d$ forming an ETF \cite{ETF,Seidel}, namely the regular simplex
of size $d+1$, the set of these vectors uniquely minimizes the $p$-frame potential for
$p> 2$.

When $p\in (0,2)$, not much is known about minimizers of this value except in a few special cases. In \cite{Ehler}, Ehler and Okoudjou solved the simplest case where $d=2$ and $N=3$ and also proved  that the minimizer of the $p$-frame potential is exactly $n$ copies of an orthonormal basis if $N=nd$ where $n$ is a positive integer. In \cite{Glazyrin1}, Glazyrin  provided a lower bound for any $1\leq p\leq2$:
\begin{equation}\label{e7}
\FP_{p,N,d}(X)\geq 2(N-d)\frac{1}{p^{\frac{p}{2}}(2-p)^{\frac{2-p}{2}}},
\end{equation}
but the condition under which  equality holds is strict.
 In \cite{Chen}, Chen, Gonzales, Goodman, Kang and Okoudjou considered this special
 case where $N=d+1$.  Particularly, \cite{Chen} showed through numerical experiments
  that the sets $L_k^d$, which they call lifted ETFs, seem to be minimizers of the
   $p$-frame potential for certain $k$ depending on $p$. Here,
   $L_k^d=\{\vx_1,\ldots,\vx_{d+1}\}\subset \R^d$ is defined as a set of $d+1$ unit vectors in $\R^d$ satisfying
\begin{equation}\label{gram}
\absinn{\vx_i,\vx_j}:=\left\{
        \begin{array}{cl}
        \frac{1}{k} & i,j\in \{1,\ldots,k+1\}, i\neq j\\
                  1 & i=j\\
          0 & else
        \end{array}
      \right. .
\end{equation}
Note that $\{\vx_i\}_{i=1}^{k+1}\subset L_k^d$ actually forms an ETF in some subspace $W\subset \R^d$ with dimension $k$ and the rest of $d-k$ vectors form an orthonormal basis in the orthogonal complement of $W$.

More precisely, the following  conjecture is proposed in  \cite{Chen}:

\begin{conjecture}\label{conj}
Suppose $d\geq2$. Set $p_0:=0$, $p_d:=2$ and $p_k:=\frac{\ln(k+2)-\ln(k)}{\ln(k+1)-\ln(k)}$ for each $k\in\{1,2,\ldots,d-1\}$. Then, when $p\in(p_{k-1},p_k]$,  $k=1,2,\ldots,d$, the set $L_k^d$ minimizes the $p$-frame potential when $N=d+1$.
\end{conjecture}

The cases $d=2$ and $p=2$ for Conjecture \ref{conj} are already solved in \cite{Ehler} and \cite{Fickus}, respectively.
The first new result for Conjecture \ref{conj} was obtained by
 Glazyrin in \cite{Glazyrin2} who showed  that an orthonormal basis in $\R^d$ plus a repeated vector minimizes
$\FP_{p,d+1,d}(X)$
for any $p\in(0,2(\frac{\ln{3}}{\ln{2}}-1)]$. Combining Glazyrin's result with the previous ones,  the minimizer of $\FP_{p,d+1,d}(X)$ is only known for $p\in(0,2(\frac{\ln{3}}{\ln{2}}-1)]\cup[2,\infty)$.
  Recently, Park extented Glazyrin's result to the case $N=d+m$ where $1\leq m<d$, and showed that an orthonormal basis plus $m$ repeated vectors is the minimizer for any $p\in[1,2\frac{\ln{(2m+1)}-\ln{(2m)}}{\ln{(m+1)}-\ln{(m)}}]$ (see \cite{Park}). But Conjecture \ref{conj} remains open when $d>2$.

\subsection{Our contributions}
The aim of this paper is to confirm Conjecture \ref{conj} and to show that  the minimizers are unique provided $p\neq p_k$.
 Our main result is the following theorem which  completely solves  the minimal $p$-frame potential problem for the case where $N=d+1$.

\begin{theorem} \label{T1}
Let $d\geq 2$ be an integer. Set $p_0:=0$, $p_d:=2$ and $p_k:=\frac{\ln(k+2)-\ln(k)}{\ln(k+1)-\ln(k)}$ for each $k\in\{1,2,\ldots,d-1\}$. Assume that $p\in(0,2)$ is a real number. Let $X=\{\vx_1,\ldots,\vx_N\}$ be a set of $N$ unit vectors in $\R^d$, where $N=d+1$.
\begin{enumerate}[{\rm (i)}]
\item
For  $p\in (p_{k-1},p_k), k=1,2,\ldots,d$, and for any $X\in  S(d+1,d) $ we have
$
\FP_{p,d+1,d}(X)\geq (k+1)k^{1-p}
$
and equality holds if and only if   $X=L_k^d$.
 \item For $p=p_k, \, k=1,\ldots, d-1$,
 and for any $X\in  S(d+1,d) $ we have
$
\FP_{p,d+1,d}(X)\geq (k+1)k^{1-p_k}
$
and equality holds if and only if
 $X=L_k^d$ or $X=L_{k+1}^d$.
\end{enumerate}

\end{theorem}

Based on the previous results and Theorem \ref{T1} in this paper,  in Table \ref{tab1}, we list optima for the minimal $p$-frame potential problem when $N=d+1$. Note that $2(\frac{\ln{3}}{\ln{2}}-1)\approx 1.16993$ and $\frac{\ln{3}}{\ln{2}}\approx 1.58496$.
Hence, $(0,2(\frac{\ln{3}}{\ln{2}}-1)]$ is a subinterval in $(0,\frac{\ln{3}}{\ln{2}})$.
In Table \ref{tab1}, we also use the fact that $L_1^d$ is essentially an orthonormal basis plus a repeated vector and $L_d^d$ forms an ETF in $\R^d$.

{\small
\makeatletter\def\@captype{table}\makeatother
\begin{center}
\caption {Minimizer of the $p$-frame potential when $N=d+1$ } \label{tab1}
\begin{tabular}{|l|l|l}
\cline{1-2}
 $p$ &  Minimizers &    \\ \cline{1-2}
$p\in(0,2(\frac{\ln{3}}{\ln{2}}-1)]$  & $L_1^d$ \cite{Glazyrin2} &  \\ \cline{1-2}
$p=2$ & $L_d^d$ \cite{Sidelnikov,Fickus}  & \\ \cline{1-2}
$p\in(2,\infty)$ & $L_d^d$ \cite{Ehler} &   \\ \cline{1-2}
$p\in(0,\frac{\ln{3}}{\ln{2}})$ & $L_1^d$ (Theorem \ref{T1}) &   \\ \cline{1-2}
$p\in\left(\frac{\ln({(k+1)}/{(k-1)})}{\ln({k}/{(k-1)})},\frac{\ln({(k+2)}/{k})}{\ln({(k+1)}/{k})}\right),k=2,3,\ldots,d-1$ & $L_k^d$ (Theorem \ref{T1}) &   \\ \cline{1-2}
$p\in(\frac{\ln({(d+1)}/{(d-1)})}{\ln({d}/{(d-1)})},2)$ & $L_d^d$ (Theorem \ref{T1}) &   \\ \cline{1-2}
$p=\frac{\ln((k+2)/k)}{\ln((k+1)/k)},\ k=1,2,\ldots,d-1$ & $L_k^d$ and $L_{k+1}^d$(Theorem \ref{T1}) &   \\ \cline{1-2}
\end{tabular}
\end{center}
}

\subsection{Organization}

The paper is organized as follows. In Section 2, we prove Theorem \ref{T1} based on Lemma \ref{M}. The proof of Lemma \ref{M} is presented in Section 3.

\section{Proof of Theorem $\ref{T1}$}

In this section we present a proof of Theorem \ref{T1}. The following lemma plays a key role in our proof of Theorem \ref{T1}. We postpone  its proof to Section 3.
To this end, we set
\[
M_{\alpha, d+1}(z_1,\ldots,z_{d+1})\,\,:=\,\,\sum\limits_{i=1}^{d+1}\sum\limits_{j\neq i}z_i^\alpha z_j^\alpha
\]
where $\alpha>1$.
We consider
\begin{equation}\label{con0}
\argmax{(z_1,\ldots,z_{d+1})} M_{\alpha,d+1}(z_1,\ldots,z_{d+1}),\quad {\rm s.t.}\quad z_1+\cdots+z_{d+1}=1,
z_1\geq0,\ldots,z_{d+1}\geq 0,
\end{equation}
where $\alpha> 1$.
 Noting that $M_{\alpha,d+1}(z_1,\ldots,z_{d+1})$ is a symmetric function on $d+1$ variables $z_1,\ldots,z_{d+1}$, we identify solutions to (\ref{con0}) only up to permutations of  $z_i$.
\begin{lemma}\label{M}
Suppose that $d\geq 1$ is an integer.
 Set
  \begin{equation}\label{eq:ak}
 a_k:=\left\{
        \begin{array}{cl}
        \infty & k=0\\
          \frac{1}{2}\cdot\frac{\ln(k+2)-\ln(k)}{\ln(k+2)-\ln(k+1)} & k\in\{1,2,\ldots,d-1\}\\
          1 & k=d
        \end{array}
      \right. .
\end{equation}
\begin{enumerate}[{\rm (i)}]
\item If   $ \alpha\in (a_{k},a_{k-1})$ then the unique solution to (\ref{con0})
is $\left(\underbrace{\frac{1}{k+1},\ldots,\frac{1}{k+1}}_{k+1},\underbrace{0,\ldots,0}_{d-k}\right)$  where $k=1,2,\ldots,d$.
\item  Assume that $\alpha=a_k$ where $k=1,\ldots,d-1$. Then (\ref{con0}) has exactly two solutions:
$\left(\underbrace{\frac{1}{k+1},\ldots,\frac{1}{k+1}}_{k+1},\underbrace{0,\ldots,0}_{d-k}\right)$ and
$\left(\underbrace{\frac{1}{k+2},\ldots,\frac{1}{k+2}}_{k+2},\underbrace{0,\ldots,0}_{d-k-1}\right)$.
\end{enumerate}

\end{lemma}

We next state a proof of Theorem \ref{T1}.  Our method of estimating the $p$-frame
potential in this proof is motivated by the work of  Bukh and Cox  \cite{Bukh}. For a
finite set $X=\{\vx_i\}_{i=1}^{N}\in S(N,d)$, Bukh and Cox derived a new lower bound
on $ \mu(X):=\max\limits_{j\neq l}\abs{\langle \vx_j,\vx_l\rangle} $ with the help of
an orthonormal basis in the null space of the  Gram matrix of $X$ (see also \cite{Dels}). We borrow their
idea of  considering the null space of the Gram matrix of $X$. Noting the
corresponding null space in our case is a one-dimensional subspace in $\R^N$, we pick
a unit vector $\vy$ in this  subspace,  showing that we can use the coordinates of
$\vy$ to provide an estimation on the value of $\FP_{p, d+1,d}(X)$.
\begin{proof}[Proof of Theorem \ref{T1}]
(i)
Note that
$
\FP_{p,d+1,d}(L_k^d)= (k+1)k^{1-p}
$.
To this end, it is enough to show that
$
\FP_{p,d+1,d}(X)\geq (k+1)k^{1-p}
$
when $p\in(p_{k-1},p_k)$ and $L_k^d$ is the unique minimizer for each $k\in\{1,2,\ldots,d\}$.

Recall that $X=\{\vx_i\}_{i=1}^{d+1}\subset \R^d$ is a set of $d+1$ unit-norm vectors. Set
\[
G=(\langle \vx_i,\vx_j \rangle)\,\,\in\,\, \R^{(d+1)\times(d+1)}.
 \]
Note that $\rank ( G)\leq d$. Thus, there exists a unit vector $\vy=(y_1,\ldots,y_{d+1})^T\in\R^{d+1}$ such that $G\vy=0$. We compute the value of $(i,i)$-entry of the matrix $G\vy\vy^T$ and obtain
\begin{equation*}
0=(G\vy\vy^T)_{i,i}=\sum\limits_{j=1}^{d+1}\langle \vx_i,\vx_j \rangle\cdot y_iy_j=y_i^2+\sum\limits_{j\neq i}\langle \vx_i,\vx_j \rangle\cdot y_iy_j,
\end{equation*}
which implies
\begin{equation*}
y_i^2=|\sum\limits_{j\neq i}\langle \vx_i,\vx_j \rangle\cdot y_iy_j|\leq \sum\limits_{j\neq i}|\langle \vx_i,\vx_j \rangle|\cdot |y_i||y_j|.
\end{equation*}
Summing up the above inequality from $1$ to $d+1$, we obtain
\begin{equation*}
1\,\,=\,\,\sum\limits_{i=1}^{d+1}y_i^2\,\,\leq\,\, \sum\limits_{i=1}^{d+1}\sum\limits_{j\neq i}|\langle \vx_i,\vx_j \rangle|\cdot |y_i||y_j|.
\end{equation*}

We next present a proof for (i)
with dividing  the  proof into two cases:

{\em Case 1: $p\in (0,1]$.}
Note that $(0,1]\subset  (p_0,p_1)$. It is enough to prove that the unique solution to
$
\argmin{X\in S(d+1,d)}\FP_{p,d+1,d}(X)
$
 is $X=L_1^d$ for any $p\in(0,1]$.
 We first consider the case where $p=1$.
Since
\begin{equation*}
|y_i||y_j |\,\,\leq\,\, \frac{y_i^2+y_j^2}{2}\,\,\leq\,\,\frac{1}{2},\,\, \text{ for all } i\neq j
\end{equation*}
we obtain
\begin{equation*}
1\leq \sum\limits_{i=1}^{d+1}\sum\limits_{j\neq i}|\langle \vx_i,\vx_j \rangle|\cdot |y_i||y_j|\leq \frac{1}{2}\cdot\sum\limits_{i=1}^{d+1}\sum\limits_{j\neq i}|\langle \vx_i,\vx_j \rangle|,
\end{equation*}
which implies
\begin{equation}\label{p1}
\sum\limits_{i=1}^{d+1}\sum\limits_{j\neq i}|\langle \vx_i,\vx_j \rangle|\,\,\geq\,\, 2.
\end{equation}
The equality in (\ref{p1}) holds if and only if there exist $i_1,i_2\in \{1,2,\ldots,d+1\}$ with $i_1\neq i_2$ such that $|\langle \vx_{i_1},\vx_{i_2} \rangle|=1$ and the remaining terms in the sum $\sum\limits_{i=1}^{d+1}\sum\limits_{j\neq i}|\langle \vx_i,\vx_j \rangle|$ are all zero. We arrive at the conclusion.

We next turn to the case $p\in(0,1)$. Noting $\absinn{ \vx_i,\vx_j }\leq 1$, we have
\begin{equation*}
|\langle \vx_i,\vx_j \rangle|^p\geq |\langle \vx_i,\vx_j \rangle|,  \ \forall i\neq j
\end{equation*}
for any $p\in(0,1)$.
Thus,
\begin{equation}\label{11}
\sum\limits_{i=1}^{d+1}\sum\limits_{j\neq i}|\langle \vx_i,\vx_j \rangle|^p\geq \sum\limits_{i=1}^{d+1}\sum\limits_{j\neq i}|\langle \vx_i,\vx_j \rangle|\geq 2.
\end{equation}
The equality holds if and only if $|\langle \vx_i,\vx_j \rangle|=0$ or 1 for any distinct $i,j$. Thus, the minimizer of $1$-frame potential is also the unique minimizer of $p$-frame potential for any $p\in(0,1)$.

{\em Case 2: $1<p<2$.}
For  $1<p< 2$,  we use H$\ddot{\text{o}}$lder's inequality to obtain
\begin{equation}\label{Holder}
1\leq \sum\limits_{i=1}^{d+1}\sum\limits_{j\neq i}|\langle \vx_i,\vx_j \rangle|\cdot |y_i||y_j|\leq (\sum\limits_{i=1}^{d+1}\sum\limits_{j\neq i}|\langle \vx_i,\vx_j \rangle|^p)^{\frac{1}{p}}(\sum\limits_{i=1}^{d+1}\sum\limits_{j\neq i}|y_i|^q|y_j |^q)^{\frac{1}{q}}
\end{equation}
where $q>2$ satisfies $\frac{1}{p}+\frac{1}{q}=1$. The second equality in (\ref{Holder}) holds if and only if there exists a constant $c\in\R$ such that
\begin{equation}\label{Holder2}
c\cdot |\langle \vx_i,\vx_j \rangle|^{p-1}\,\, =\,\,|y_i||y_j|, \text{ for all } i\neq j.
\end{equation}
Equation (\ref{Holder}) implies
\begin{equation}\label{13}
\FP_{p,d+1,d}(X)\geq \frac{1}{(\sum\limits_{i=1}^{d+1}\sum\limits_{j\neq i}|y_i|^q|y_j |^q)^{\frac{p}{q}}}.
\end{equation}

Let $\alpha=\frac{q}{2}$ and $z_i=|y_i|^2$ for  $i=1,2,\ldots,d+1$. Then we
 can rewrite the inequality in (\ref{13}) as
\begin{equation}\label{eq:eq1}
\FP_{p,d+1,d}(X)\,\,\geq\,\, \frac{1}{(M_{\alpha,d+1}(z_1,\ldots,z_{d+1}))^{\frac{p}{q}}},
\end{equation}
where $M_{\alpha,d+1}(z_1,\ldots,z_{d+1})=\sum\limits_{i=1}^{d+1}\sum\limits_{j\neq i}z_i^{\alpha}z_j^{\alpha}$, $z_1+\cdots+z_{d+1}=1, z_i\geq 0, i=1,\ldots,d+1$.

Note that $\alpha=\frac{q}{2}=\frac{1}{2}+\frac{1}{2}\frac{1}{p-1}$. If $p\in(p_{k-1},p_k)\cap(1,2)$ where $k\in\{1,\ldots,d\}$, then $\alpha\in (a_{k},a_{k-1})$. Here, $a_k$ is defined in (\ref{eq:ak}). According to Lemma \ref{M}, $M_{\alpha,d+1}(z_1,\ldots,z_{d+1})$ attains its maximum, $k(k+1)^{1-2\alpha}$, only when $z_i=\frac{1}{k+1}$ for  $i=1,\ldots,k+1$ and $z_i=0$ for $i\geq k+2$. Thus, we obtain
\begin{equation}\label{14}
\FP_{p,d+1,d}(X) \,\,\geq\,\, \frac{1}{(k(k+1)^{1-2\alpha})^{\frac{p}{q}}}=(k+1)k^{1-p}
\end{equation}
when $p\in(p_{k-1},p_k)\cap(1,2)$, $k=1,\ldots,d$. Combining with equation (\ref{Holder2}), the equality in (\ref{14}) holds if and only if
for $i\neq j$
\begin{equation}\label{eq1:ak}
 \absinn{\vx_i,\vx_j}=\left\{
        \begin{array}{cl}
        \frac{1}{k} & i,j\in \{1,\ldots,k+1\}\\
         0 & else
        \end{array}
      \right. ,
\end{equation}
which implies that $X=L_k^d$.
Combining the analysis in {\em Case 1}, we arrive at the conclusion (i) for Theorem \ref{T1}.

(ii) Note that
$
\FP_{p,d+1,d}(L_k^d)=\FP_{p,d+1,d}(L_{k+1}^d)= (k+1)k^{1-p}
$
when $p=p_k$, $k=1,2,\ldots, d-1$. To this end, it is enough to prove that $\FP_{p_k,d+1,d}(X) \,\,\geq\,\, (k+1)k^{1-p_k}$ and that the minimizers are $L_{k}^d$ and $L_{k+1}^d$. Since $p_k\in(1,2)$ for each $k\in\{1,2,\ldots,d-1\}$, we follow our analysis for the proof of (i).

If $p=p_k$ where $k\in\{1,\ldots,d-1\}$, then $\alpha$ in (\ref{eq:eq1}) is equal to $a_k$. According to Lemma \ref{M}, $M_{a_k,d+1}(z_1,\ldots,z_{d+1})$ attains its maximum, which is $k(k+1)^{1-2a_k}$, at exactly two points:
$\left(\underbrace{\frac{1}{k+1},\ldots,\frac{1}{k+1}}_{k+1},\underbrace{0,\ldots,0}_{d-k}\right)$ and
$\left(\underbrace{\frac{1}{k+2},\ldots,\frac{1}{k+2}}_{k+2},\underbrace{0,\ldots,0}_{d-k-1}\right)$. Thus, we obtain
\begin{equation}\label{144}
\FP_{p_k,d+1,d}(X) \,\,\geq\,\, \frac{1}{(k(k+1)^{1-2a_k})^{\frac{p_k}{2\cdot a_k}}}=(k+1)k^{1-p_k}
\end{equation}
for $k=1,2,\ldots,d-1$.
According to  (\ref{Holder2}), the equality in (\ref{144}) holds if and only if $X=L_k^d$ or $L_{k+1}^d$, which implies the conclusion (ii).

\end{proof}

\begin{remark}
For convenience, we state Theorem \ref{T1} and its proof for the real case.
In fact, it is easy to extend the result in Theorem \ref{T1} to complex case.
Moreover, the method which is employed to prove Theorem \ref{T1}  can be used to estimate the matrix potential, i.e. $\sum\limits_{i\neq j}|A_{i,j}|^p$, where $A_{i,j}$ is the $(i,j)$-entry of any  matrix $A\in \C^{(d+1)\times (d+1)}$ whose rank is $d$ and diagonal elements are equal to 1 (see \cite{Glazyrin2}).
\end{remark}

\section{Proof of Lemma $\ref{M}$}

In this section we present the proof of  Lemma \ref{M}.
We begin with introducing  the following lemma, which portrays the main feature of local extrema for (\ref{con0}).

For convenience, we set
\begin{equation}\label{fm}
f_{m_1,\alpha,d+1}(t):=M_{\alpha,d+1}\left(\underbrace{t,\ldots,t}_{m_1},\underbrace{s,\ldots,s}_{d+1-m_1}
\right)=(m_1\cdot t^{\alpha}+(d+1-m_1)\cdot s^{\alpha})^2-(m_1\cdot t^{2\alpha}+(d+1-m_1)\cdot s^{2\alpha}),
\end{equation}
where $s:=\frac{1-m_1t}{d+1-m_1}, m_1\in [1,\frac{d+1}{2}]\cap \Z$.

\begin{lemma}\label{L1}
Assume that $(w_1,\ldots,w_{d+1})$ is a local maxima of  $M_{\alpha,d+1}(z_1,\ldots,z_{d+1})$ subject to the constraints in (\ref{con0}) and $w_i>0$ for each $i\in\{1,2,\dots,d+1\}$. Then
\begin{enumerate}[{\rm (i)}]
\item The maxima $(w_1,\ldots,w_{d+1})$ is in the form
$\left(\underbrace{t_0,\ldots,t_0}_{m_1},\underbrace{s_0,\ldots,s_0}_{d+1-m_1}\right)$ up to a permutation where $m_1\in [1,\frac{d+1}{2}]\cap \Z, t_0\in (0,\frac{1}{m_1})$ and $s_0=\frac{1-m_1t_0}{d+1-m_1}$.
\item  The value $t_0$ is a local maxima of $f_{m_1,\alpha,d+1}(t)$.
\end{enumerate}
\end{lemma}
\begin{proof}
(i) We claim that  $w_1,\ldots,w_{d+1}$ can take at most two different values. Note that
$M_{\alpha,d+1}(z_1,\ldots,z_{d+1})$  is a symmetric function on $z_1,\ldots,z_{d+1}$. Hence, up to a permutation, we can write $(w_1,\ldots,w_{d+1})$ as $\left(\underbrace{t_0,\ldots,t_0}_{m_1},\underbrace{s_0,\ldots,s_0}_{d+1-m_1}\right)$ for some $t_0\in(0,\frac{1}{m_1})$ and $s_0=\frac{1-m_1t_0}{d+1-m_1}$. It remains to prove the claim.
Set $r_0(z_1,\ldots,z_{d+1}):=z_1+\cdots+z_{d+1}-1$ and
\[
r_i(z_1,\ldots,z_{d+1}):=-z_i,\quad i=1,2,\ldots,d+1.
\]
Since $(w_1,\ldots,w_{d+1})$ is a local extreme point, according to  KKT conditions,  there exist constants $\lambda$ and $\mu_i, i=1,2,\ldots,d+1$, which are called KKT multipliers, such that the followings hold:
\begin{subequations}
\begin{align}
\nabla M_{\alpha,d+1}(w_1,\ldots,w_{d+1})&=\lambda\nabla r_0(w_1,\ldots,w_{d+1})+\sum\limits_{i=1}^{d+1}\mu_i\nabla r_i(w_1,\ldots,w_{d+1})\label{kkt1}\\
r_0(w_1,\ldots,w_{d+1})&=0\label{kkt2}\\
r_i(w_1,\ldots,w_{d+1})&\leq 0,  i=1,2,\ldots,d+1\label{kkt3}\\
\mu_ir_i(w_1,\ldots,w_{d+1})&=0,  i=1,2,\ldots,d+1\label{kkt4}\\
\mu_i&\geq 0, i=1,2,\ldots,d+1.\label{kkt5}
\end{align}
\end{subequations}

Combining $w_i>0$ and  (\ref{kkt4}), we can obtain  $\mu_i=0, i=1,2,\ldots,d+1$.
Substituting $\mu_i=0$ into (\ref{kkt1}), we obtain
\begin{equation}\label{23}
2\alpha\cdot w_i^{\alpha-1}((w_1^\alpha+\cdots+w_{d+1}^\alpha)-w_i^\alpha)=\lambda, \quad i=1,\ldots,d+1,
\end{equation}
which implies that $\lambda>0$ and
\[
\frac{\lambda}{2\alpha w_i^{\alpha-1}}+w_i^\alpha=w_1^\alpha+\cdots+w_{d+1}^\alpha,\quad i=1,\ldots,d+1.
\]
Hence, we obtain
\begin{equation}\label{eq:fw}
f(w_1)\,=\,f(w_2)\,=\,\cdots\,=\,f(w_{d+1})\,>\,0
\end{equation}
where $f(x):=x^\alpha+\frac{\lambda}{2\alpha}\cdot \frac{1}{x^{\alpha-1}}$.
Set $w_0:=(\frac{\alpha-1}{2\alpha^2}\cdot \lambda)^{\frac{1}{2\alpha-1}}$.
 Noting that $f'(x)=\alpha x^{\alpha-1}-\frac{\lambda (\alpha-1)}{2\alpha}x^{-\alpha}$,
 we obtain that $f'(x)<0, x\in (0,w_0)$, $f'(w_0)=0$ and $f'(x)>0, x\in (w_0,\infty)$,
 which implies that, for any $c\in \R$, the  cardinality of the set $\{x : f(x)=c, x>0\}$ is less than or equal to $2$ .
  Hence, equation (\ref{eq:fw}) implies that  $w_1,\ldots,w_{d+1}$ can take at most two different values.

(ii) Combining
\[
f_{m_1,\alpha,d+1}(t)=M_{\alpha,d+1}\left(\underbrace{t,\ldots,t}_{m_1},\underbrace{s,\ldots,s}_{d+1-m_1}
\right)
\]
with $\left(\underbrace{t_0,\ldots,t_0}_{m_1},\underbrace{s_0,\ldots,s_0}_{d+1-m_1}\right)$ being a local maxima of $M_{\alpha,d+1}\left(\underbrace{t,\ldots,t}_{m_1},\underbrace{s,\ldots,s}_{d+1-m_1}\right)$, we obtain the conclusion immediately.
\end{proof}

\begin{lemma}\label{L4}
 Let $m_1\in[1,\frac{d+1}{2}]\cap \Z$ and $m_2=d+1-m_1$ where $d\geq 2$ is an integer. Set
\[
h(x):=(m_2-1)x^{4\alpha-2}-m_2\cdot x^{2\alpha}+m_1\cdot x^{2\alpha-2}-(m_1-1)
\]
where $\alpha>1$. We use $h'(x)$ to denote the derivative with respect to $x$ of
$h(x)$.
 Then
\begin{enumerate}[{\rm (i)}]
\item The function $h'(x)$ has at most two zeros on $(0,\infty)$, and hence $h(x)$ has at most two extreme points on $(0,\infty)$;
\item  If $\alpha< 1+\frac{1}{d-1}$, then there exist $\hat{x}_1\in(0,1)$, $\hat{x}_2\in(1,\infty)$ such that $h'(x)>0$ for $x\in (0,\hat{x}_1)\cup(\hat{x}_2,\infty)$ and $h'(x)<0$ for $x\in (\hat{x}_1,\hat{x}_2)$;
\item If $\alpha\geq 1+\frac{1}{d-1}$, then $h(x)$ is positive and monotonically increasing on $(1,\infty)$;

\item If $\alpha=1+\frac{1}{d-1}$ and $m_1=\frac{d+1}{2}$, then $h(x)$ is monotonically increasing on $(0,\infty)$;

\item If $\alpha=1+\frac{1}{d-1}$ and $m_1<\frac{d+1}{2}$, then there exists $\hat{x}_3\in(0,1)$ such that $h'(x)>0$ for  $x\in (0,\hat{x}_3)\cup(1,\infty)$ and $h'(x)<0$ for $x\in (\hat{x}_3,1)$.

\end{enumerate}

\end{lemma}

\begin{proof}
We split the proof, proving each claim separately.

(i) By computation, we have
\begin{equation}\label{26}
h'(x)=h_1(x)\cdot x^{2\alpha-3},
\end{equation}
where $h_1(x)=(4\alpha-2)\cdot (m_2-1)x^{2\alpha}-2\alpha\cdot m_2\cdot x^2+(2\alpha-2)\cdot m_1$.

Set
\begin{equation}\label{x0}
x_0:=\left(\frac{m_2}{(2\alpha-1)\cdot (m_2-1)}\right)^{\frac{1}{2\alpha-2}}.
\end{equation}
Noting that $h_1'(x)<0, x\in (0,x_0)$, $h_1'(x)>0, x\in (x_0,\infty)$ and $h_1'(x_0)=0$, $h_1(x)=0$ has at most two distinct  solutions on $(0,\infty)$. According to (\ref{26}),  $h'(x)=0$ also has at most two distinct solutions on $(0,\infty)$, which implies the conclusion.

(ii) When $\alpha< 1+\frac{1}{d-1}$, we obtain $h_1(1)=2\alpha(d-1)-2d<0$. Then we have
\begin{equation}\label{hh1}
\inf\limits_{x>0} h_1(x)=h_1(x_0)\leq h_1(1)<0.
\end{equation}
Observing that $m_2>1$ and $\alpha>1$, we obtain
\begin{equation}\label{hh2}
h_1(0)=(2\alpha-2)\cdot m_1>0
\end{equation}
\begin{equation}\label{hh3}
\lim\limits_{x\to+\infty} h_1(x)=+\infty
\end{equation}
Thus, combining (\ref{hh1}), (\ref{hh2}) and (\ref{hh3}), we  obtain that $h_1(x)=0$ has exactly two solutions $\hat{x}_1$, $\hat{x}_2$, where $\hat{x}_1\in(0,1)$, $\hat{x}_2\in(1,\infty)$. By the monotonicity of $h_1(x)$, we also know that $h_1(x)<0$, $x\in(\hat{x}_1,\hat{x}_2)$ and $h_1(x)>0$, $x\in(0,\hat{x}_1)\cup(\hat{x}_2,\infty)$. According to (\ref{26}),  we obtain that $h'(x)<0$, $x\in(\hat{x}_1,\hat{x}_2)$  and $h'(x)>0$, $x\in(0,\hat{x}_1)\cup(\hat{x}_2,\infty)$.

(iii)
Note that
\begin{equation}
x_0\,=\,\left(\frac{1}{(2\alpha-1)\cdot (1-\frac{1}{m_2})}\right)^{\frac{1}{2\alpha-2}}\,\leq\,\left(\frac{1}{(1+\frac{2}{d-1})\cdot (1-\frac{2}{d+1})}\right)^{\frac{1}{2\alpha-2}} \,=\,1
\end{equation}
where we use $m_2=d+1-m_1\geq\frac{d+1}{2}$ and  $\alpha\geq 1+\frac{1}{d-1}$.
 So the function $h_1(x)$ is monotonically increasing when $x>1$. Noting that $h_1(1)=2\alpha(d-1)-2d\geq 0$, we have $h_1(x)>0$ on $(1,\infty)$, which implies that $h(x)$ is monotonically increasing on $(1,\infty)$. Since $h(1)=0$, we conclude that $h(x)>0$ when $x>1$.

(iv)
When $\alpha=1+\frac{1}{d-1}$ and $m_1=\frac{d+1}{2}$, we have $h_1(1)=0$ and $x_0=1$  from (\ref{x0}), which implies that $h_1(x_0)=0$. Since $x_0$ is the minimum point of $h_1(x)$, we obtain $h_1(x)\geq 0$ on $(0,\infty)$. Finally, from (\ref{26}) we see that $h'(x)\geq 0$ on $(0,\infty)$, which implies the conclusion.

(v)
Noting that $x_0\neq 1$ provided  $\alpha=1+\frac{1}{d-1}$ and $m_1<\frac{d+1}{2}$, we have
\begin{equation}\label{hhh1}
\inf\limits_{x>0} h_1(x)=h_1(x_0)< h_1(1)=2\alpha(d-1)-2d=0.
\end{equation}
Since $\alpha=1+\frac{1}{d-1}$, from (iii) we have that $h(x)$ is monotonically increasing on $(1,\infty)$. Noting that (\ref{hh2}) and (\ref{hh3}) also hold for $\alpha=1+\frac{1}{d-1}$, we conclude that $h_1(x)=0$ has exactly two solutions $\hat{x}_3$ and 1, where $\hat{x}_3\in(0,1)$. From (\ref{26}),  we obtain that $h'(x)<0$, for $x\in(\hat{x}_3,1)$ and $h'(x)>0$, for $x\in(0,\hat{x}_3)\cup(1,\infty)$.

\end{proof}

We next study the local maxima of  $f_{m_1,\alpha,d+1}(t)$ for each $m_1\in[1,\frac{d+1}{2}]\cap \Z$ and $\alpha\in (1,1+\frac{1}{d-1}]$.  The following lemma shows that if $1<\alpha\leq1+\frac{1}{d-1}$, then $f_{m_1,\alpha,d+1}(t)$ attains a local maximum at $t_0$ only if $t_0\in\{0,\frac{1}{d+1},\frac{1}{m_1}\}$.

\begin{lemma}\label{L2}

Assume $d\geq 2$ is an integer and $m_1\in[1,\frac{d+1}{2}]\cap \Z$.

\begin{enumerate}[{\rm (i)}]
\item Assume that $1<\alpha<1+\frac{1}{d-1}$, $t_0\in [0,\frac{1}{m_1}] $
and $f_{m_1,\alpha,d+1}(t)$ has a local maximum at $t_0$. Then
$
t_0\in \left\{ 0, \frac{1}{d+1}, \frac{1}{m_1}\right\}.
$

\item Assume that $\alpha=1+\frac{1}{d-1}$, $t_0\in [0,\frac{1}{m_1}] $
and $f_{m_1,\alpha,d+1}(t)$ has a local maximum at $t_0$. Then
$
t_0\in \left\{ 0,  \frac{1}{m_1}\right\}.
$
\end{enumerate}

\end{lemma}

\begin{proof}

For convenience, let $m_2:=d+1-m_1>1$. Recall that
\begin{equation}
f_{m_1,\alpha,d+1}(t)=M_{\alpha,d+1}\left(\underbrace{t,\ldots,t}_{m_1},\underbrace{s,\ldots,s}_{m_2}
\right)=(m_1\cdot t^{\alpha}+m_2\cdot s^{\alpha})^2-(m_1\cdot t^{2\alpha}+m_2\cdot s^{2\alpha}),
\end{equation}
where  $s=\frac{1-m_1\cdot t}{m_2}$. Noting that $t,s\geq 0$ and $m_1\cdot t+m_2\cdot s=1$, we can set $t=\frac{\cos^2\theta}{m_1}$, $s=\frac{\sin^2\theta}{m_2}$, where $\theta\in[0,\frac{\pi}{2}]$. We use the substitution $t=\frac{\cos^2\theta}{m_1}, s=\frac{\sin^2\theta}{m_2}$ to transform the function from $f_{m_1,\alpha,d+1}(t)$ to
\begin{equation*}
g(\theta):=f_{m_1,\alpha,d+1}\left(\frac{\cos^2\theta}{m_1}\right)=\frac{m_1(m_1-1)}{m_1^{2\alpha}}(\cos\theta)^{4\alpha}+\frac{m_2(m_2-1)}{m_2^{2\alpha}}(\sin\theta)^{4\alpha}+\frac{2m_1m_2}{m_1^{\alpha}m_2^{\alpha}}(\cos\theta\sin\theta)^{2\alpha}.
\end{equation*}
To this end, it is enough to study the local maxima of $g$ on $[0,\frac{\pi}{2}]$.
A simple calculation shows that
\begin{equation}\label{g1}
\begin{aligned}
g'(\theta)=&-4\alpha\cdot\frac{m_1(m_1-1)}{m_1^{2\alpha}}(\cos\theta)^{4\alpha-1}\sin\theta+4\alpha\cdot\frac{m_2(m_2-1)}{m_2^{2\alpha}}(\sin\theta)^{4\alpha-1}\cos\theta\\
&+2\alpha\cdot\frac{2m_1m_2}{m_1^{\alpha}m_2^{\alpha}}(\cos\theta\sin\theta)^{2\alpha-1}(\cos^2\theta-\sin^2\theta).
\end{aligned}
\end{equation}
 We can rewrite $g'(\theta)$ as
\begin{equation}\label{g2}
g'(\theta)=4\alpha\cdot \frac{m_1}{m_1^{2\alpha}}\cdot (\cos\theta)^{4\alpha-1}\sin\theta\cdot h(v),
\end{equation}
where $v:=\sqrt{\frac{s}{t}}=\sqrt{\frac{m_1}{m_2}}\cdot \frac{\sin\theta}{\cos\theta}$ and $h(v):=(m_2-1)v^{4\alpha-2}-m_2\cdot v^{2\alpha}+m_1\cdot v^{2\alpha-2}-(m_1-1)$.
Particularly, when $\theta=\theta_*:=\arctan(\sqrt{\frac{m_2}{m_1}})$, we have $v=\sqrt{\frac{m_1}{m_2}}\cdot \frac{\sin\theta_*}{\cos\theta_*}=1$.

Noting  that $\alpha>1$, $m_1\geq 1$ and $m_2>1$,  we obtain
\[
h(0)=-(m_1-1)\leq0\label{h1},\,\,
h(1)=0,\,\,
\lim\limits_{v\to+\infty}h(v)=+\infty.
\]
Since $4\alpha\cdot \frac{m_1}{m_1^{2\alpha}}\cdot (\cos\theta)^{4\alpha-1}\sin\theta$ is positive for any $\theta\in(0,\frac{\pi}{2})$, to study the monotonicity of $g(\theta)$,
 it is enough to consider the sign of $h(v)$ with $v>0$.

(i)
First we consider the case $1<\alpha<1+\frac{1}{d-1}$.

Lemma $\ref{L4}$ shows that there exists $\hat{v}_1\in(0,1)$ and $\hat{v}_2\in(1,\infty)$ such that
$h'(v)>0$ for $v\in (0,\hat{v}_1)\cup(\hat{v}_2,\infty)$ and $h'(v)<0$ for $v\in (\hat{v}_1,\hat{v}_2)$.
Noting that $h(1)=0$ and $\hat{v}_1<1<\hat{v}_2$, we obtain that $h(\hat{v}_1)>0$ and $h(\hat{v}_2)<0$.
Combining Lemma \ref{L4} and the results above,  we obtain that $h(v)=0$ has exactly one solution on $[0,\hat{v}_1)$ , say $v_1$. Similarly,   $h(v)=0$ also has exactly one solution on $(\hat{v}_2,\infty)$, say   $v_2$. Let $\theta_1:=\arctan(v_1\sqrt{\frac{m_2}{m_1}})$ and $\theta_2:=\arctan(v_2\sqrt{\frac{m_2}{m_1}})$.

If $m_1=1$, then we have $h(0)=0$ and hence  $v_1=0$. From the monotonicity of $h(v)$, we obtain that $h(v)<0$, $v\in(1,v_2)$, $h(v)>0$, $v\in(0,1)\cup({v}_2,\infty)$ and $h(v)=0$, $v\in\{0,1,v_2\}$. Then from (\ref{g2}) it is easy to check that $g'(\theta)<0$, $\theta\in(\theta_*,\theta_2)$, $g'(\theta)>0$, $\theta\in(0,\theta_*)\cup(\theta_2,\frac{\pi}{2})$ and $g'(\theta)=0$, $\theta\in\{0,\theta_*,\theta_2,\frac{\pi}{2}\}$, which implies $g(\theta)$ has only two local maxima: $\theta_*$ and $\frac{\pi}{2}$.

If $m_1>1$, then $h(0)<0$, which means $v_1\in(0,\hat{v}_1)$.  Thus, by the monotonicity of $h(v)$ we conclude that $h(v)<0$, $v\in(0,v_1)\cup(1,v_2)$, $h(v)>0$, $v\in({v}_1,1)\cup({v}_2,\infty)$ and $h(v)=0$, $v\in\{v_1,1,v_2\}$.
We can use (\ref{g2}) to transform these results to $g'(\theta)$.
Hence, we obtain that $g'(\theta)<0$, $\theta\in(0,\theta_1)\cup(\theta_*,\theta_2)$, $g'(\theta)>0$, $\theta\in(\theta_1,\theta_*)\cup(\theta_2,\frac{\pi}{2})$ and $g'(\theta)=0$, $\theta\in\{0,\theta_1,\theta_*,\theta_2,\frac{\pi}{2}\}$, which implies $g(\theta)$ has only three local maxima: 0, $\theta_*$ and $\frac{\pi}{2}$.

(ii)
We next  consider the case where $\alpha=1+\frac{1}{d-1}$. We divided the proof into two cases.

{\em Case 1: $m_1=\frac{d+1}{2}$ .}
 Lemma $\ref{L4}$ implies that $h(v)$ is monotonically increasing on $(0,\infty)$. Noting that $h(0)=-(m_1-1)<0$ and $h(1)=0$, we have $h(v)<0$, $v\in(0,1)$ and $h(v)>0$, $v\in(1,\infty)$. We use (\ref{g2}) to transform the result to $g'(\theta)$ and obtain   that $g'(\theta)<0$, $\theta\in(0,\theta_*)$, $g'(\theta)>0$, $\theta\in(\theta_*,\frac{\pi}{2})$ and $g'(\theta)=0$, $\theta\in\{0,\theta_*,\frac{\pi}{2}\}$, which implies $g(\theta)$ has only two local maxima: 0 and $\frac{\pi}{2}$.

{\em Case 2: $m_1<\frac{d+1}{2}$ .}
According to  Lemma $\ref{L4}$,  there exists $\hat{v}_3\in(0,1)$ such that $h'(v)>0$ for  $v\in (0,\hat{v}_3)\cup(1,\infty)$ and $h'(v)<0$ for  $v\in (\hat{v}_3,1)$.

If $m_1=1$, then $h(0)=h(1)=0$. According to the sign of $h'(v)$, we obtain that  $h(v)\geq 0$, $v\in[0,\infty)$. Equation (\ref{g2}) implies that $g'(\theta)$ is always non-negative on $[0,\frac{\pi}{2}]$, which means $\frac{\pi}{2}$ is the only local maxima of $g(\theta)$.

If $1<m_1<\frac{d+1}{2}$, then $h(0)<0$. So there exists  $v_3\in(0,\hat{v}_3)$ such that $h(v)<0$, $v\in(0,v_3)$ and $h(v)\geq0$, $v\in[v_3,\infty)$. Set $\theta_3:=\arctan(v_3\sqrt{\frac{m_2}{m_1}})$. According to (\ref{g2}), we have $g'(\theta)<0$, $\theta\in(0,\theta_3)$, $g'(\theta)>0$, $\theta\in(\theta_3, \frac{\pi}{2})$ and $g'(\theta)=0$, $\theta\in\{0,\theta_3,\frac{\pi}{2}\}$, which implies $g(\theta)$ has only two local maxima: 0 and $\frac{\pi}{2}$.

\end{proof}

\begin{remark}
When $1<\alpha\leq1+\frac{1}{d-1}$, combining Lemma \ref{L1} and Lemma \ref{L2}, we obtain that $(\frac{1}{d+1},\ldots,\frac{1}{d+1})$ is the only local maxima of $M_{\alpha,d+1}(z_1,\ldots,z_{d+1})$ with the constraints $z_1+\cdots+z_{d+1}=1$ and $z_i> 0$, $i=1,2,\ldots,d+1$.

\end{remark}

We deal with the case $\alpha>1+\frac{1}{d-1}$ in the next lemma.

\begin{lemma}\label{L3}
Assume that $\alpha>1+\frac{1}{d-1}$ and $d\geq 2$. Assume that $(w_1,w_2,\ldots,w_{d+1})$ is a local maxima of  $M_{\alpha,d+1}(z_1,\ldots,z_{d+1})$ with the constraints in (\ref{con0}).
Then there exists $k_0\in \{1,\ldots,d+1\}$ such that $w_{k_0}=0$.
\end{lemma}

\begin{proof}
We proceed by contradiction, supposing that $w_i>0$ for $i\in \{1,\ldots,d+1\}$.
According to Lemma \ref{L1}, $(w_1,\ldots,w_{d+1})$ is in the form
$\left(\underbrace{t_0,\ldots,t_0}_{m_1},\underbrace{s_0,\ldots,s_0}_{d+1-m_1}\right)$ up to a permutation where $m_1\in [1,\frac{d+1}{2}]\cap \Z$, $t_0\in (0,\frac{1}{m_1})$ and $s_0=\frac{1-m_1t_0}{d+1-m_1}$.
Lemma \ref{L1} also implies that  $t_0$ is a local maxima of $f_{m_1,\alpha,d+1}(t)$. So, it is enough to show the following claim:

{\em Claim 1:} When $\alpha>1+\frac{1}{d-1}$, if $t_0\in (0,\frac{1}{m_1})$ is a local maxima
of $f_{m_1,\alpha,d+1}(t)$, then $\left(\underbrace{t_0,\ldots,t_0}_{m_1},\underbrace{s_0,\ldots,s_0}_{d+1-m_1}\right)$ is not a local maxima of $M_{\alpha,d+1}(z_1,\ldots,z_{d+1})$ with the constraints in (\ref{con0}).

Claim 1 contradicts $\left(\underbrace{t_0,\ldots,t_0}_{m_1},\underbrace{s_0,\ldots,s_0}_{d+1-m_1}\right)$ being  a local maxima of $M_{\alpha,d+1}(z_1,\ldots,z_{d+1})$ with the constraints in (\ref{con0}). Hence, there exists $k_0\in\{1,\ldots,d+1\}$ such that $w_{k_0}=0$.

It remains to prove Claim 1.
For convenience, set $m_2:=d+1-m_1$. Since $m_1\leq \frac{d+1}{2}$ and $d\geq 2$, we have $m_2\geq2$.
Set
\begin{equation*}
F(\varepsilon):=M_{\alpha,d+1}\left(\underbrace{t_0,\ldots,t_0}_{m_1},s_0+l\varepsilon,\underbrace{s_0-\varepsilon,\ldots,s_0-\varepsilon}_{m_2-1}\right),
\end{equation*}
where $l=m_2-1$ and $\varepsilon\in (-\frac{s_0}{l},s_0)$. We aim to show that $\varepsilon=0$ is not a local maxima of $F(\varepsilon)$. In fact, we can prove this by showing that  $\varepsilon=0$ is a local minima of $F(\varepsilon)$.

A simple calculation shows that
\begingroup\fontsize{10pt}{12pt}\selectfont
\[
\begin{aligned}
F(\varepsilon)=&(m_1\cdot t_0^\alpha+(s_0+l\varepsilon)^\alpha+(m_2-1)(s_0-\varepsilon)^\alpha)^2-(m_1\cdot t_0^{2\alpha}+(s_0+l\varepsilon)^{2\alpha}+(m_2-1)(s_0-\varepsilon)^{2\alpha}),\\
F'(\varepsilon)=&2\alpha\cdot (m_1\cdot t_0^\alpha+(s_0+l\varepsilon)^\alpha+(m_2-1)(s_0-\varepsilon)^\alpha)\cdot (l(s_0+l\varepsilon)^{\alpha-1}-(m_2-1)(s_0-\varepsilon)^{\alpha-1})\\
&-2\alpha\cdot (l(s_0+l\varepsilon)^{2\alpha-1}-(m_2-1)(s_0-\varepsilon)^{2\alpha-1}),\\
F''(\varepsilon)=&2\alpha^2\cdot (l(s_0+l\varepsilon)^{\alpha-1}-(m_2-1)(s_0-\varepsilon)^{\alpha-1})^2\\
&+2\alpha(\alpha-1)\cdot (m_1\cdot t_0^\alpha+(s_0+l\varepsilon)^\alpha+(m_2-1)(s_0-\varepsilon)^\alpha)\cdot(l^2(s_0+l\varepsilon)^{\alpha-2}+(m_2-1)(s_0-\varepsilon)^{\alpha-2}) \\
&-2\alpha(2\alpha-1)\cdot (l^2(s_0+l\varepsilon)^{2\alpha-2}+(m_2-1)(s_0-\varepsilon)^{2\alpha-2}).
\end{aligned}
\]
\endgroup
Noting $l=m_2-1$, we can check that
\begin{equation}\label{f'}
F'(0)=0.
\end{equation}
We claim $F''(0)>0$ and hence $\varepsilon=0$ is a local minima of $F(\varepsilon)$.

It remains finally to prove $F''(0)>0$.
Note that
\begin{equation}\label{t0}
\begin{aligned}
F''(0)&=2\alpha\cdot (l^2+m_2-1)\cdot s_0^{\alpha-2}((\alpha-1)(m_1t_0^\alpha+m_2s_0^\alpha)-(2\alpha-1)s_0^\alpha).
\end{aligned}
\end{equation}
Since  $t_0\notin \{ 0,\frac{1}{m_1}\}$ is a local maxima of $f_{m_1,\alpha,d+1}(t)$, from equation (\ref{g2}) we know that $\sqrt{\frac{s_0}{t_0}}$ is a root of $h(v)=0$, where $h(v)=(m_2-1)v^{4\alpha-2}-m_2\cdot v^{2\alpha}+m_1\cdot v^{2\alpha-2}-(m_1-1)$.
According to Lemma   \ref{L4}, $h(v)>0$ for $v>1$ provided $\alpha\geq 1+\frac{1}{d-1}$, which
 implies that $\sqrt{\frac{s_0}{t_0}}\leq 1$ and hence $s_0\leq t_0$. Combining $s_0>0$ and  $l^2+m_2-1\geq 2$, we have
\[
\begin{aligned}
F''(0)&\geq 2\alpha\cdot (l^2+m_2-1)\cdot s_0^{\alpha-2}((\alpha-1)(m_1s_0^\alpha+m_2s_0^\alpha)-(2\alpha-1)s_0^\alpha)\\
&=2\alpha\cdot (l^2+m_2-1)\cdot s_0^{2\alpha-2}((\alpha-1)(m_1+m_2)-(2\alpha-1))\\
&=2\alpha\cdot (l^2+m_2-1)\cdot s_0^{2\alpha-2}((d-1)\alpha-d).
\end{aligned}
\]
Noting that $\alpha>1+\frac{1}{d-1}$, we obtain
\begin{equation}\label{f''}
F''(0)>0.
\end{equation}
\end{proof}

We next present the proof of Lemma \ref{M}.

\begin{proof}[Proof of Lemma \ref{M}]

 We prove Lemma \ref{M}  by induction on $d$. First, we consider the case $d=1$. For $d=1$, we have only two non-negative variables $z_1,z_2$ which satisfy $z_1+z_2=1$. For any $\alpha> 1$ we have
\begin{equation*}
M_{\alpha,2}=2z_1^{\alpha} z_2^{\alpha}\leq 2\cdot \left(\frac{z_1+z_2}{2}\right)^{\alpha}=2^{1-\alpha},
\end{equation*}
where equality holds if and only if $z_1=z_2=\frac{1}{2}$. Hence, the solution to (\ref{con0}) is $(\frac{1}{2},\frac{1}{2})$ which implies  Lemma \ref{M} holds for $d=1$. We assume that Lemma \ref{M} holds for $d= d_0-1$ and hence we know the solution to (\ref{con0})
 for $d=d_0-1$. So, we consider the case where $d=d_0$.

Assume that  $(w_1,\ldots,w_{d_0+1})$ is a solution to (\ref{con0}) with $d=d_0$.  Recall that
$a_0=\infty$, $a_{d_0}=1$, $a_k=\frac{1}{2}\cdot\frac{\ln(k+2)-\ln(k)}{\ln(k+2)-\ln(k+1)}$, $k=1,2,\ldots,d_0-1$. For convenience, we set
$\be_{k+1}:=(\frac{1}{k+1},\ldots,\frac{1}{k+1})\in \R^{k+1}$ and $\bl_{d_0-k}:=(0,\ldots,0)\in \R^{d_0-k}$.
We also set
\[
(\be_{k+1},\bl_{d_0-k}):=\left(\underbrace{\frac{1}{k+1},\ldots,\frac{1}{k+1}}_{k+1},
\underbrace{0,\ldots,0}_{d_0-k}\right).
\]
We first show 
\begin{equation}\label{eq:jihe}
(w_1,\ldots,w_{d_0+1})\in \left\{\left(\be_{k+1},\bl_{d_0-k}\right)\in \R^{d_0+1}:k=1,\ldots,d_0\right\},
\end{equation}
dividing the proof into two cases.

{\em Case 1: $\alpha\in (1+\frac{1}{d_0-1},\infty)$ .}

According to Lemma \ref{L3}, at least one of the entries in $(w_1,\ldots,w_{d_0+1})$ is 0.  Without loss of generality, we assume $w_{d_0+1}=0$.  Since $M_{\alpha,d_0+1}(w_1,\ldots,w_{d_0},0)=M_{\alpha,d_0}(w_1,\ldots,w_{d_0})$,  $(w_1,\ldots,w_{d_0})$ is the solution to (\ref{con0}) with $d=d_0-1$. Hence, by induction we conclude that (\ref{eq:jihe}) holds.

{\em Case 2: $\alpha\in (1,1+\frac{1}{d_0-1}]$.}

If one of entries in $(w_1,\ldots,w_{d_0+1})$ is 0, we can show that
(\ref{eq:jihe}) holds using a similar argument as above.
So, we consider the case where  $w_i>0$ for each $i\in\{1,\ldots,d_0+1\}$. Lemma \ref{L1} shows that
 $
 (w_1,\ldots,w_{d_0+1})$ is in the form  $\left(\underbrace{t_0,\ldots,t_0}_{m_1},\underbrace{s_0,\ldots,s_0}_{d_0+1-m_1}\right)
 $ up to a permutation where $m_1\in [1,\frac{d_0+1}{2}]\cap \Z$, $t_0\in (0,\frac{1}{m_1})$ and $s_0=\frac{1-m_1t_0}{d_0+1-m_1}$
. Lemma \ref{L1} also implies that  $t_0$ is a local maxima of the function $f_{m_1,\alpha,d_0+1}(t)$, where $f_{m_1,\alpha,d_0+1}(t)$ is defined in (\ref{fm}). According to  Lemma \ref{L2}, we obtain $t_0=\frac{1}{d_0+1}$. Hence $(w_1,\ldots,w_{d_0+1})=(\frac{1}{d_0+1},\ldots,\frac{1}{d_0+1})$ , which implies (\ref{eq:jihe}).

It is now enough to  compare the values among
$
M_{\alpha,d_0+1}\left(\be_{k+1},\bl_{d_0-k}\right),
k=1,\ldots,d_0.
$
Setting $H(x):=x^{1-2\alpha}(x-1)$, we obtain $M_{\alpha,d_0+1}\left(\be_{k+1},\bl_{d_0-k}\right)=H(k+1)$ for each $k\in\{1,2,\ldots,d_0\}$. A simple calculation shows that $H(x)$ is monotonically increasing on $(0,1+\frac{1}{2\alpha-2})$ and monotonically decreasing on $(1+\frac{1}{2\alpha-2},\infty)$. Hence, the sequence
 $H(k+1),k=1,\ldots,d_0$, is unimodal.

 (i) Firstly, we consider the case where $\alpha\in(a_k,a_{k-1})$, $k=1,2,\ldots,d_0-1$.
 Noting that $H(k)<H(k+1)$ and $H(k+1)>H(k+2)$, we obtain
\begin{equation}
\max\limits_{x\in\{1,2,\dots,d_0\}} H(x+1)=H(k+1), \text{ for all } \alpha\in(a_k,a_{k-1}),
\end{equation}
where equality holds if and only if $x=k$. For the case $\alpha\in(a_{d_0},a_{d_0-1})$, noting that $H(d_0)<H(d_0+1)$, we obtain
\begin{equation}
\max\limits_{x\in\{1,2,\dots,d_0\}} H(x+1)=H(d_0+1), \text{ for all } \alpha\in(a_{d_0},a_{d_0-1}),
\end{equation}
To summarize,
$
\left(\be_{k+1},\bl_{d_0-k}\right)
 $
 is the unique solution to (\ref{con0}) with $d=d_0$ when $\alpha\in(a_k,a_{k-1})$, $k=1,2,\ldots,d_0$.

(ii) It remains finally to check the case where $\alpha=a_k$, $k=1,2,\ldots,d_0-1$.
Noting $H(k+1)=H(k+2)$, $H(k)<H(k+1)$ and $H(k+2)>H(k+3)$ provided $\alpha=a_k,\ k=1,2,\ldots,d_0-2$,
we obtain that (\ref{con0}) has two solutions which are
 $\left(\be_{k+2},\bl_{d_0-k-1}\right)$ and $\left(\be_{k+1},\bl_{d_0-k}\right)$  with $d=d_0$ .
When $\alpha=a_{d_0-1}$, noting $H(d_0)=H(d_0+1)$ and $H(d_0-1)<H(d_0)$, we obtain that (\ref{con0}) also has two solutions which are
 $\be_{d_0+1}$ and $\left(\be_{d_0},\bl_{1}\right)$  with $d=d_0$ .
Hence, the conclusion also holds for $d=d_0$ which completes the proof.
\end{proof}

{\bf Acknowledgement.} We would  like to thank the anonymous reviewers for their
comments which help improve this paper substantially.

\end{document}